\documentclass{amsart}

\usepackage{amssymb, amsthm}
\usepackage{hyperref}
\usepackage{xcolor}
\usepackage[T2A]{fontenc}
\usepackage[utf8x]{inputenc}
\usepackage{setspace}
\usepackage[english]{babel}
\usepackage{cmap}
\usepackage{amsmath}
\usepackage{amsfonts}
\usepackage{graphicx}
\usepackage{amssymb}
\usepackage{multicol,caption}
\usepackage[final]{listings}
\usepackage{makecell}
\usepackage{tikz}
\usepackage{cite}
\usepackage[noabbrev]{cleveref}
\usepackage{bbold}
\usepackage{enumitem}

\newcommand{\dist}{\mathop{\rm dist}\nolimits}

\newcommand{\id}{\mathop{\rm Id}\nolimits}
\newcommand{\D}{\mathop{\rm D}\nolimits}

\newcommand{\R}{\mathbb{R}}
\newcommand{\N}{\mathbb{N}}
\newcommand{\Z}{\mathbb{Z}}
\newcommand{\p}{\mathbb{P}}

\newcommand{\T}{\mathbb{T}}
\newcommand{\E}{\mathbb{E}}

\newcommand{\dd}{\mbox{d}}
\newcommand{\emr}[1]{\textcolor{black}{#1}}

\theoremstyle{plain}
\newtheorem{thm}{Theorem}
\newtheorem{lm}{Lemma}

\theoremstyle{definition}
\newtheorem*{defn}{Definition}

\newtheorem{cor}{Corollary}

\theoremstyle{remark}
\newtheorem{rem}{Remark}

\title{Probabilistic shadowing in linear skew products}

\author{Grigorii V. Monakov}
\address{Department of Mathematics, University of California, Irvine, CA~92697, USA.} 
\email[Corresponding author]{gmonakov@uci.edu.} 
\thanks{G. M. was supported in part by NSF grant DMS--2247966 (PI: A.\,Gorodetski).}

\author{Sergey B. Tikhomirov}
\address{Pontificia Universidade Catolica do Rio de Janeiro, Rua Marques de Sao Vicente, 225, Gavea, Rio de Janeiro, RJ - Brasil CEP: 22451-900; Instituto de Matematica Pura e Applicada, Estrada Dona Castorina, 110, Jardim Botanico, Rio de Janeiro, RJ - Brasil, CEP: 22460-320.} \email{sergey.tikhomirov@gmail.com}
\thanks{S. T. was supported by FAPERJ PDS 2021, process code E-26/202.311/2021 (261921), Projeto Paz and Coordenacao de Aperfeicoamento de Pessoal de Nivel Superior - Brasil (CAPES) - Finance Code 001}



\begin{document}

    \begin{abstract}

        We investigate the probability of shadowing of a random finite pseudotrajectory by an exact trajectory for linear skew products.
        We describe general conditions under which a random pseudotrajectory can be shadowed with polynomial (with respect to its length) precision with high probability.
        Examples satisfying that general condition are continuous linear skew products over Bernoulli shift, doubling map on a circle, and any Anosov linear map on a torus. The main tool used in the proof is Cramer's large deviation theorem.
    \end{abstract}

    \maketitle

\section{Introduction}


Shadowing of pseudotrajectories is a well-developed area of the qualitative theory of dynamical systems. There are a lot of results that show the connection between the shadowing property and structural stability. First of all, we would mention celebrated shadowing lemma by D. Anosov and R. Bowen, stating that a diffeomorphism has shadowing property near hyperbolic set \cite{Anosov, Bowen}. It is also well known that if a diffeomorphism is structurally stable on the whole manifold then it has shadowing property \cite{Robinson, Sawada}. See books \cite{Pil, Palm} for a comprehensive survey. Although shadowing property does not imply structural stability (for some examples see \cite{FinHolSh}), this implication is true with some additional assumptions \cite{Abdenur, Sakai, LipSh, FinHolSh}. For a review of recent results on relation between shadowing and structural stability see \cite{SakPil}

At the same time, results of various numerical experiments show that finite pseudotrajectories can be shadowed effectively even for systems that are not structurally stable. For example, S. Hammel, C. Grebogy and J. Yorke using interval arithmetic showed that for logistic map and H\'enon map a typical $d$-pseudotrajectory of length $\sim 1/\sqrt{d}$ can be $\sim \sqrt{d}$ shadowed \cite{Hammel1, Hammel2}. At this time no general statement of this sort is known. 

It is an important question what type of shadowing can be expected in dynamical systems that are not uniformly hyperbolic. In the present paper we study the stochastic setting of the shadowing problem for pseudotrajectories. For the first time stochastic shadowing was introduced in \cite{YorkeYuan}. Later it was shown that considering infinite pseudotrajectories in stochastic setting gives results equivalent to the classical shadowing theory \cite{nonshad}. We introduce a natural way of generating a finite random pseudotrajectory and estimate the probability of it to be shadowed by an exact one. For the first time this approach was suggested in \cite{LinProd} by the second author for random i.i.d. variables. 

Another approach for shadowing was introduced by A. Katok in his celebrated paper \cite{Kat}. In this approach size of allowed pseudotrajectory errors depends on the index of Pesin set of the point, which allows to establish shadowing and closing lemmas for non-uniformly hyperbolic systems. Note that in the approach discussed in present paper sizes of pseudotrajectory errors are uniform and do not depend on the point. See also \cite{PPT} where a probabilistic approach was considered for nonuniform sizes of pseudotrajectory errors.

 Closely related to the probabilistic shadowing is the notion of stochastic stability. 
 Dynamical system is \textit{stochastically stable} if the stationary measure for the randomly perturbed dynamics is close  to the invariant measure of the original system. This property encodes important information on statistics of dynamical behaviour. It is known that stochastic stability holds for hyperbolic systems  \cite{Kif86, Kif88, You86} and for a large class of non uniformly hyperbolic systems \cite{Via97, AlA03, BeV06, AAV07, AlVil13, Al06, Blank}. While notions of stochastic stability and probabilistic shadowing both rely on similarities of random pseudotrajectories  and exact trajectories there is no known formal relation between them.


It was shown in works \cite{LipSh, FinHolSh, Monakov, PerSh, Mor2023} that the parameters of the shadowing property for finite pseudotrajectories can be obtained using the asymptotics of the growth rate of the solution of the corresponding inhomogeneous linear system. In this work we study nonuniformly hyperbolic linear skew products. Under certain assumptions we show that there exists a polynomial condition on value of the error, accuracy of shadowing, and the length of a pseudotrajectory, such that the probability that a random pseudotrajectory can be shadowed by an exact trajectory tends to one as the length tends to infinity. Next, we show that our theorem is applicable for three fundamental chaotic examples: skew products over Bernoulli shift, doubling map on a circle and linear Anosov map on a torus. 


The proof we present is based on the Cramér's large deviations theorem that is known to be true for dynamical systems with various rates of mixing \cite{DK1, DK2, CDK}. The last fact suggests that our approach might be used for establishing the shadowing property in a more general setting. {It is also worth noticing that some connections were established between mixing properties, large deviations and stochastic stability \cite{BK, BV, G}. 
Even though a formal connection between these phenomena and probabilistic shadowing has not been found yet, we believe that they are close in nature.}


\section{Statement of the main result}
\label{chap:1}

Consider a compact metric space $(X, \dist_X)$, a continuous map
\begin{equation*}
    T: X \to X,
\end{equation*}
and a Borel probability measure $\nu$ on $X$. For future purposes we will require that every ball $B$ in $X$ has positive measure $\nu$. Let us also consider a continuous function $\lambda: X \to \R_{+}$ ($\R_+ = (0, +\infty)$) such that
\begin{gather}
    \label{lambda_exp}
    \E_{\nu}(\log(\lambda)) = \int_{X} \log(\lambda(x)) \dd \nu(x) \ne 0.
\end{gather}
Define the space $Q$ by the equality $Q = X \times \R$ and consider the measure $m = \nu \times \mathrm{Leb}$ and maximum metric:

\begin{equation*}
    \dist ((w, x), (\Tilde{w}, \Tilde{x})) = \max (\dist_X(w, \Tilde{w}), |x - \Tilde{x}|).
\end{equation*}
We consider the skew product map $f : Q \to Q$ that is given by the following formula:
\begin{gather} \label{skewProdDef}
    f(w, x) = (T(w), \lambda(w) x).
\end{gather}

\emr{In further considerations $T$ would be a Bernoulli shift, or uniformly expanding map, or uniformly hyperbolic map. In that case map $f$ resembles a nonuniformly partially hyperbolic skew product with $X$ being the (hyperbolic or expanding) base and $\R$ representing a one-dimen\-sional central direction. The condition (\ref{lambda_exp}) is analogous to having a nonzero Lyapunov exponent.}


\begin{defn}
    We call the sequence of points $\{y_k\}_{k = a}^b$ a $d$-pseudotrajectory for the map $f:Q \to Q$ if 
    \begin{equation*}
        \dist(y_{k+1}, f(y_k)) < d, \quad k \in \{a, \dots, b - 1\}.
    \end{equation*}
\end{defn}
\begin{defn}
    We say that trajectory $\{x_k\}_{k = a}^b$ $\varepsilon$-shadows the pseudotrajectory $\{y_k\}$ if
    
    \begin{equation*}
        \dist(y_k, x_k) < \varepsilon, \quad k \in \{a, \dots, b\}.
    \end{equation*}
\end{defn}

\emr{For every $q \in Q$, every radius $d > 0$, and every natural $N \in \N$ we denote by $\Omega_{q, d, N}$ the set of all $d$-pseudotrajectories of the length $N$ that start at the point $q$. }

\emr{
In order to define notion of random pseudotrajectory we will use construction of Markov chains \cite{Revuz}. Let us denote by $B(q, r)$ the open ball of radius $r$ in the space $Q$ centered at the point $q$ and define a Markov kernel $K_{d}$ (see \cite[Definition 5.1]{DK1} and \cite[Definition 1.8]{Revuz}) on $Q$ using the following formula:
\begin{equation} \label{kernel}
    K_{d} (x, A) = \frac{m(A \cap B(f(x), d))}{m(B(f(x), d))} \quad \text{for every $x \in Q$ and every Borel $A \subset Q$.}
\end{equation}
Following notations from \cite[Chapter 2]{Revuz} (with $\Omega = Q$ and $\mathcal{F}$ being the Borel $\sigma$-algebra on $Q$) we set the initial distribution $P_0 = \delta_{q}$ -- delta measure at a fixed $q \in Q$. Together with the transition probability $K_d$ that gives rise to a Markov chain $\{\xi^{(d)}_j\}_{j = 0}^{\infty}$ on $Q$. It follows immediately from the definition of $K_{d}$ that, almost every realization of $(\xi^{(d)}_0, \xi^{(d)}_1, \ldots, \xi^{(d)}_{N})$ lies in $\Omega_{q, d, N}$. Hence, the distribution of $(\xi^{(d)}_0, \xi^{(d)}_1, \ldots, \xi^{(d)}_{N})$ is supported on $\Omega_{q, d, N}$. For a positive number $\varepsilon > 0$ let us denote by $p (q, d, N, \varepsilon)$ the probability that a pseudotrajectory from $\Omega_{q, d, N}$ chosen at random with respect to the distribution of $(\xi^{(d)}_0, \xi^{(d)}_1, \ldots, \xi^{(d)}_{N})$ can be $\varepsilon$-shadowed by an exact trajectory. The above construction is a generalization of the systems that were studied in \cite{LinProd, YorkeYuan}.
}



\emr{In order to introduce a natural initial distribution for the Markov chain described above we need the following lemma about shift invariance.} We will use Lemma 1 from \cite{LinProd} without repeating its proof because the original one works in our setting without any changes.
\begin{lm}
    \label{unif_lemma}
    Consider two points $q = (w, x)$ and $\Tilde{q} = (w, 0)$. For arbitrary positive numbers $d, \varepsilon > 0$, $N \in \N$ the following equality holds:
    \begin{gather*}
        p(q, d, N, \varepsilon) = p(\Tilde{q}, d, N, \varepsilon).
    \end{gather*}
\end{lm}

For positive numbers $d, \varepsilon > 0$, $N \in \N$ let us define
\begin{gather*}
    p(d, N, \varepsilon) = \int_{w \in X} p((w, 0), d, N ,\varepsilon) d\nu.
\end{gather*}

\emr{In other words, we consider initial distribution $P_0 = \nu \times \delta_0$ and form a new Markov chain $\{\xi^{(d)}_j\}_{j = 0}^{\infty}$ using the same kernel $K_{d}$ defined by (\ref{kernel}). We will denote said Markov chain by $\mathbb{M}_d$. From now on we call a random $d$-pseudotrajectory a random realization of $\mathbb{M}_{d}$. We will denote by $\p$ the corresponding probability. }

To continue we will need to introduce three important properties of the basic map~$T$.

\begin{defn}[Property \textbf{I}]
    \emr{ We  say that the map $T$ satisfies Property \textbf{I}, if there exist constants $L, d_0 > 0$ such that for every $N \in \N$ and every $d < d_0$ there exists a function $w: X^{N + 1} \to X$, such that given a random $d$-pseudotrajectory $(\xi^{(d)}_0, \xi^{(d)}_1, \ldots, \xi^{(d)}_{N})$, where $\xi^{(d)}_j = (w_j, x_j)$ we have
    \begin{gather*}
        dist_X(w_j, T^j(w((w_j)_{j = 0}^{N}))) < L d,
    \end{gather*}
    and $w ((w_j)_{j = 0}^{N})$ has distribution $\nu$.}

\end{defn}

\begin{defn}[Property \textbf{II}]
    Let us take a random point $w$ in $X$ distributed with respect to measure $\nu$ and introduce 
    \begin{gather*}
        A_j' = \sum_{p = 0}^{j - 1} \log \left( \lambda(T^p (w)) \right).
    \end{gather*}
    We  say that the map $f$ satisfies Property \textbf{II}, if there exist constants $C, k > 0$, such that for every $j \in \N$ the following holds:
    \emr{\begin{gather}
        \label{property_2}
        \nu \left( \bigg \{ w \in X : \left| \dfrac{A_j'}{j} - \E_{\nu}(\log(\lambda)) \right| > \frac{|\E_{\nu}(\log(\lambda))|}{2} \bigg \} \right) < C e^{- k j}.
    \end{gather}}
\end{defn}

\begin{defn}[Property \textbf{III}]
    \emr{We say that the map $T$ and measure $\nu$ satisfy Property \textbf{III} if there exists $d_0$ such that for any $d < d_0$ the measure $\nu$ is invariant with respect to the projection of $K_d$ on the first coordinate, namely
    \begin{equation*}
        \nu(A) = \int_{X} K_{d} \left( (w, x), \; A \times \R \right) \; \dd \nu(w) \quad \text{for every $x \in \R$ and every Borel $A \subset X$.}
    \end{equation*}}

\end{defn}

Now we can state the main theorem of this work:
\begin{thm}
    \label{skew_prod_thm}
    If a continuous function $\lambda: X \to \R_{+}$ satisfies assumption~(\ref{lambda_exp}) and the map $f$ satisfies Properties \textbf{I}, \textbf{II} and \textbf{III}, then there exist positive numbers $\varepsilon_0 > 0$ and  $\gamma > 0$ such that for every $0 < \varepsilon < \varepsilon_0$ the following equality holds:
    \begin{gather}
        \label{main_thm_ineq}
        \lim_{N \to \infty} p \left( \dfrac{\varepsilon}{N^\gamma}, N, \varepsilon \right) = 1.
    \end{gather}
\end{thm}


Although the conditions in Theorem \ref{skew_prod_thm} might seem rather restrictive, in Section \ref{applications} we will use this result to establish probabilistic shadowing for three fundamental uniformly hyperbolic examples of dynamical system $T$. The examples are the following:
\begin{enumerate}[label=(\Alph*)]
    \item
    \emph{Bernoulli shift.} \label{BS}
    Let $X = \{0, 1\}^{\Z}$ with standard metric and uniform probability measure and $T:X \to X$ -- the left shift;
    
    \item
    \emph{Doubling map.} \label{DM}
    Let $X = \R / \Z \cong \T^1$ -- unit circle with standard metric and Lebesgue measure and $T:X \to X$ given by formula 
    \begin{gather*}
        T(w) = 2w \mod 1;
    \end{gather*}
    
    \item
    \emph{Linear Anosov map on a Torus.} \label{AM}
    Let $X = \R^m / \Z^m \cong \T^m$ -- flat $m$-dimensional torus with standard metric and Lebesgue measure and $T:X \to X$ given by formula 
    \begin{gather*}
        T(w) = Aw \mod 1,
    \end{gather*}
    where $A$ is a hyperbolic matrix with integer entries and determinant $\pm1$.
\end{enumerate}

In Section \ref{applications} we will prove the following 
\begin{thm} \label{examplesThm}
    Let $(X, \dist_X, T, \nu)$ be one of the systems \ref{BS}, \ref{DM} or \ref{AM} and $\lambda: X \to \R_{+}$ be any positive continuous function satisfying (\ref{lambda_exp}). Consider a skew product $f: Q \to Q$, defined by formula (\ref{skewProdDef}). Then $f$ satisfies Properties \textbf{I}, \textbf{II} and \textbf{III}.
\end{thm}

Proofs for these three systems might look similar, but they do contain fundamental differences. Combining theorems \ref{skew_prod_thm} and \ref{examplesThm} we obtain the following

\begin{cor}
    If the function $\lambda: X \to \R_{+}$ satisfies formula~(\ref{lambda_exp}) then the skew product $f$ over systems \ref{BS}, \ref{DM} or \ref{AM} satisfy equality (\ref{main_thm_ineq}).
\end{cor}

\begin{rem}
    Note that a relatively weak regularity condition is imposed on function $\lambda$ -- it only needs to be continuous.
\end{rem} 


\section{Proof of Theorem \ref{skew_prod_thm}}

We will need some auxiliary lemmas to prove the theorem. First of all, let us consider the sequences $(\lambda_j)_{j = 0}^{N - 1}$ and $(r_j)_{j = 1}^N$ such that $|r_j| < 1$. We introduce following notations:
\begin{gather*}
    \tilde{A}_j = \sum_{p = 0}^{j - 1} \log (\lambda_p),\\
    z_0 = 0, \quad z_{j + 1} = z_j + \dfrac{r_j}{e^{A_{j + 1}}}.
\end{gather*}
and define
\begin{gather*}
    B(p, q) = \dfrac{e^{\tilde{A}_p + \tilde{A}_q}}{e^{\tilde{A}_p} + e^{\tilde{A}_q}} |z_p - z_q|; \\
    F \left( (\tilde{A}_j)_{j = 0}^{N}, (r_j)_{j = 1}^N \right) = \max_{0 \le p < q \le N} B(p, q).
\end{gather*}
The following lemma is a reformulation of Lemma 2 in \cite{LinProd}. The proof is the same so we will not repeat it.
\begin{lm} \label{F_lemma}
    For every sequence $(x_j)_{j = 0}^N$, given by 
    \begin{gather*}
        x_0 = 0, \quad
        x_{j + 1} = \lambda_j x_j + r_{j + 1},
    \end{gather*}
    There exists $y_0$ such that the sequence defined by
    \begin{gather*}
        y_{j + 1} = \lambda_j y_j,
    \end{gather*}
    satisfies inequalities
    \begin{gather*}
        |x_j - y_j| < F\left((\tilde{A}_j)_{j = 0}^{N}, (r_j)_{j = 1}^N\right), \quad \forall j \in \{0, \dots, N\}.
    \end{gather*}
\end{lm}

Let 
$
    a = \E_{\nu} (\log(\lambda))
$.
Below we will consider two cases:

\subsection{Case 1: $a < 0$} \label{sec:aneg}

We fix constants $k$ and $C$ such that the inequality (\ref{property_2}) holds. Moreover, we fix the following $\gamma_0$:
\begin{gather}
    \label{gamma_ineq}
    {\gamma_0} = 1 + (1 + \delta)\frac{\max(|\log(\lambda)|)}{k},
\end{gather}
where $\delta$ is an arbitrary positive number.

Function $\log(\lambda)$ is continuous on a compact set $X$, so there exists a positive number $d_1$, such that for $x, y \in X$ if $\dist(x, y) < d_1$ then 
\begin{gather}
    \label{continuity}
    \left| \log \left( \lambda(x) \right) - \log \left( \lambda(y) \right) \right| < -\frac{a}{2}.
\end{gather}

Now we are ready to prove Theorem \ref{skew_prod_thm}. Let $(w_j, x_j)_{j = 0}^{N}$ be a random $d$-pseudotrajectory with $d < \min(d_0, \frac{d_1}{L})$, with $d_0$ and $L$ satisfying Property \textbf{I} and Property \textbf{III} and $d_1$ defined above. We introduce the following notations:
\begin{gather}
    \label{def_A_j}
    A_j = \sum_{p = 0}^{j - 1} \log (\lambda(w_p));\\
    \label{def_r_j}
    r_j = \dfrac{x_j - \lambda(w_{j - 1}) x_{j - 1}}{d}.
\end{gather}
We will start with an estimate for the probability of the following event:
\begin{gather}\label{eqS1}
    S_1 = \{\text{$\exists \ y_0$ such that for $y_j = \lambda(w_{j - 1}) y_{j - 1}$ we have $|x_j - y_j| \le d (N + 1)^{\gamma_0}$}\}.
\end{gather}
It is clear that $|r_j| < 1$ and according to Lemma \ref{F_lemma} 
\begin{gather}
    \label{S1_reform}
    \p (S_1) \ge \p \left( F \left( (A_j)_{j = 0}^{N}, (r_j)_{j = 1}^N \right) \le (N + 1)^{\gamma_0} \right).
\end{gather}
It is easy to see that
\begin{gather}
    \label{B_ineq}
    e^{A_q} |z_q - z_p| \le \sum_{j = p + 1}^q e^{-(A_j - A_q)}
\end{gather}
and
\begin{gather}
    \label{old_ineq}
    \dfrac{e^{A_p}}{e^{A_p} + e^{A_q}} \le 1.
\end{gather}
Using (\ref{S1_reform}), (\ref{B_ineq}) and (\ref{old_ineq}) we obtain
\begin{gather}
    \label{old_ineq_2}
    1 - \p (S_1) \le \p \left( \exists\ 0 \le p < q \le N : \sum_{j = p + 1}^q e^{-(A_j - A_q)} > (N + 1)^{{\gamma_0}} \right).
\end{gather}
Measure $\nu$ is invariant with respect to the projection of the transition operator of $\mathbb{M}_d$ on $X$ (Property \textbf{III}), thus, taking $n = N - (q - (p + 1))$ we get
\begin{gather*}
    \p \left( \exists\ 0 \le p < q \le N : \sum_{j = p + 1}^q e^{-(A_j - A_q)} > (N + 1)^{{\gamma_0}} \right) \le \\
    \le (N + 1) \p \left(\exists\ n \le N : \sum_{j = n}^N e^{- (A_j - A_N)} > (N + 1)^{{\gamma_0}} \right).
\end{gather*}
After straightforward computations we obtain
\begin{gather*}
    (N + 1) \p \left(\exists\ n \le N : \sum_{j = n}^N e^{- (A_j - A_N)} > (N + 1)^{{\gamma_0}} \right) \le \\
    \le (N + 1) \p \left( \sum_{j = 0}^N e^{-(A_j - A_N)} > (N + 1)^{{\gamma_0}} \right) 
    \le (N + 1) \sum_{j = 0}^N \p \left( e^{A_N - A_j} > (N + 1)^{{\gamma_0} - 1} \right) = \\
    = (N + 1) \sum_{j = 0}^N \p \left( e^{A_j} > (N + 1)^{{\gamma_0} - 1} \right) 
    \le (N + 1) \sum_{j = 0}^N \p \left( A_j > ({\gamma_0} - 1) \log(N + 1) \right) \le \\
    \le (N + 1) \sum_{j = 0}^N \p \left( \left|\dfrac{A_j}{j} - a \right| > \dfrac{({\gamma_0} - 1) \log(N + 1)}{j} - a \right).
\end{gather*}

According to (\ref{gamma_ineq}) we know that 
\begin{gather*}
    {\gamma_0} - 1 = (1 + \delta)\frac{\max(|\log(\lambda)|)}{k}.
\end{gather*}
It follows immediately that for $j < (1 + \delta) \frac{\log(N + 1)}{k}$ we have
\begin{gather*}
    \p \left( \left|\dfrac{A_j}{j} - a \right| > \dfrac{({\gamma_0} - 1) \log(N + 1)}{j} - a \right) = 0,
\end{gather*}
because for such $j$ we have
\begin{gather*}
    \dfrac{({\gamma_0} - 1) \log(N + 1)}{j} - a > \dfrac{({\gamma_0} - 1) \log(N + 1) k }{(1 + \delta) \log(N + 1)} - a >  \max(|\log(\lambda)|) - a \ge \left|\dfrac{A_j}{j} - a \right|.
\end{gather*}
Using this fact we drop out the first $(1 + \delta) \frac{\log(N + 1)}{k}$ summands and get the following:

\begin{gather}
    \label{killed_summands}
    1 - \p (S_1) \le (N + 1) \sum_{j > (1 + \delta) \frac{\log(N + 1)}{k}}^N \p \left( \left|\dfrac{A_j}{j} - a \right| > \dfrac{({\gamma_0} - 1) \log(N + 1)}{j} - a \right).
\end{gather}

Let us denote by $w$ the point $w \left( (w_j)_{j = 0}^N \right)$ from Property \textbf{I}. We would like to replace $A_j$ with $A_j'$ to be able to use Property \textbf{II}. Just to remind the notation:
\begin{gather}
    \label{def_A_j'}
    A_j' = \sum_{p = 0}^{j - 1} \log \left( \lambda(T^p (w)) \right).
\end{gather}
We know that 
\begin{gather*}
    dist(w_j, T^j(w)) < L d < d_1,
\end{gather*}
and using (\ref{continuity}) we get 
\begin{gather*}
    |\log(\lambda(w_j)) - \log(\lambda(T^j(w))| < -\frac{a}{2}.
\end{gather*}
It easily follows that 
\begin{gather}
    \label{est_A-A'}
    \left| \frac{A_j}{j} - \frac{A_j'}{j} \right| < -\frac{a}{2},
\end{gather}
and this inequality gives us 
\begin{gather}
    \label{replaced_A_j}
    \p \left( \left|\dfrac{A_j}{j} - a \right| > \dfrac{({\gamma_0} - 1) \log(N + 1)}{j} - a \right) \le \p \left( \left|\dfrac{A_j'}{j} - a \right| > \dfrac{({\gamma_0} - 1) \log(N + 1)}{j} - \frac{a}{2} \right).
\end{gather}

Now using (\ref{killed_summands}) and (\ref{replaced_A_j}) and applying the inequality (\ref{property_2}) we obtain the following:

\begin{gather*}
    1 - \p (S_1) \le (N + 1) \sum_{j > (1 + \delta) \frac{\log(N + 1)}{k}}^N \p \left( \left|\dfrac{A_j'}{j} - a \right| > \dfrac{({\gamma_0} - 1) \log(N + 1)}{j} - \frac{a}{2} \right) \le \\
    \le (N + 1) \sum_{j > (1 + \delta) \frac{\log(N + 1)}{k}}^N \p \left( \left|\dfrac{A_j'}{j} - a \right| > - \frac{a}{2} \right) 
    \le (N + 1) \sum_{j > (1 + \delta) \frac{\log(N + 1)}{k}}^N C e^{-k j} = \\
    = C (N + 1) e^{-k (1 + \delta) \frac{\log(N + 1)}{k}} \sum_{j = 0}^{N - (1 + \delta)\frac{\log(N + 1)}{k}} e^{-k j} 
    \le C (N + 1)^{-\delta} \frac{1}{1 - e^{-k}}.
\end{gather*}

Now it is easy to see that (\ref{gamma_ineq}) implies 
\begin{gather*}
    \p(S_1) \to_{N \to \infty} 1.
\end{gather*}
For the next step we assume that $S_1$ holds. We fix $y_0$ defined in $S_1$ and introduce 
\begin{gather}
    \label{z_0}
    z_0 = y_0.
\end{gather} 
We would like to remind the reader  that we use $w = w \left( (w_j)_{j = 0}^N \right)$ from Property \textbf{II} and the distribution of $w$ equals to $\nu$. We also use the following notation:
\begin{gather}
    \label{true_gamma}
    \gamma = 2 {\gamma_0},\\
    \label{y_j}
    y_{j + 1} = \lambda(w_j) y_j,\\
    \label{z_j}
    z_{j + 1} = \lambda(T^j (w)) z_j.
\end{gather}
We would like to estimate probability of the following event: 
\begin{gather}\label{eqS2}
    S_2 = \bigg \{ \max_{j \in \{0, \dots, N\}} |z_j - y_j| > 2 d (N + 1)^{\gamma} \bigg \}.
\end{gather}
Using the fact that $|y_0| = |y_0 - x_0| \le d (N + 1)^{\gamma_0}$ and formulas (\ref{def_A_j}), (\ref{z_0}), (\ref{y_j}), (\ref{z_j}) and (\ref{def_A_j'}) we get:
\begin{gather*}
    |z_j - y_j| = |y_0| \cdot \left| e^{A_j} - e^{A_j'} \right| \le d (N + 1)^{\gamma_0} \left( e^{A_j} + e^{A_j'} \right).
\end{gather*}
Hence :
\begin{gather*}
    \p(S_2) = \p \left( \max_{j \in \{0, \dots, N\}} |z_j - y_j| > 2 d (N + 1)^{\gamma} \right) \le \\
    \le \p \left( \exists j \in \{0, \dots, N\} :\ (e^{A_j} + e^{A_j'}) > 2 (N + 1)^{\gamma_0} \right) \le \\
    \le \p \left( \exists j \in \{0, \dots, N\} :\ A_j > {\gamma_0} \log(N + 1)  \right) 
    + \p \left( \exists j \in \{0, \dots, N\} :\ A_j' > {\gamma_0} \log(N + 1) \right) \le \\
    \le \sum_{j = 0}^N \p \left( \left|\dfrac{A_j}{j} - a \right| > \dfrac{{\gamma_0} \log(N + 1)}{j} - a \right) 
    + \sum_{j = 0}^N \p \left( \left|\dfrac{A_j'}{j} - a \right| > \dfrac{{\gamma_0} \log(N + 1)}{j} - a \right).
\end{gather*}
Now we only need to estimate the two summands from previous line. Using the same estimations we have just used to obtain that $\p(S_1) \to_{N \to \infty} 1$, it is easy to see that
\begin{gather*}
    \p \left( S_2 \right) \to_{N \to \infty} 0.
\end{gather*}
To finish the proof we take $d = \dfrac{\varepsilon}{3 (N + 1)^{\gamma}}$ and obtain that for a random $d$-pseudotrajecto\-ry $(w_j, x_j)_{j = 0}^{N}$ the sequence $(y_j)_{j = 0}^{N}$ given by the equation (\ref{y_j}) will satisfy
\begin{gather*}
    |y_j - x_j| < d (N + 1)^{\gamma_0}
\end{gather*}
with probability not less than $\p(S_1)$. We also know that with probability not less than $1 - \p(S_2)$ there exist an exact trajectory $(T^j(w), z_j)_{j = 0}^{N}$ given by (\ref{z_j}) and satisfying the inequalities
\begin{gather*}
    |z_j - y_j| < 2 d (N + 1)^{\gamma}.
\end{gather*}
We have obtained that the trajectory $(T^j(w), z_j)_{j = 0}^{N}$ $\varepsilon$-shadows the pseudotrajectory $(w_j, x_j)_{j = 0}^{N}$ (since $d (N + 1)^{\gamma_0} + 2 d (N + 1)^{\gamma} < 3 d (N + 1)^{\gamma} = \varepsilon$) with probability not less than $1 - (1 - \p(S_1) + \p(S_2)) = \p(S_1) - \p(S_2) \to_{N \to \infty} 1$. To be precise, now we have proved that 
\begin{gather}\label{eq:lim3N}
    \lim_{N \to \infty} p \left( \dfrac{\varepsilon}{3 N^{\gamma}}, N, \varepsilon \right) = 1,
\end{gather}
which differs from the statement of the Theorem \ref{skew_prod_thm} by the factor $3$ in the first argument. Let us choose ${\gamma_0}'$ that lies between ${\gamma_0}$ and $1 + (1 + \delta)\frac{\max(|\log(\lambda)|)}{k}$ (and $\gamma' = 2 \gamma_0'$). Since the inequality 
\begin{gather*}
    \dfrac{\varepsilon}{N^{{\gamma}'}} < \dfrac{\varepsilon}{3 N^{\gamma}}
\end{gather*}
holds starting from some $N \in \N$ the relation \eqref{eq:lim3N} implies \eqref{main_thm_ineq} for every
${\gamma} > 2 \left( 1 + \frac{\max(|\log(\lambda)|)}{k} \right)$, 
which finishes the proof for the case $a < 0$.

\subsection{Case 2: $a > 0$}

In this subsection we will deal with 
\begin{gather*}
    a = \E_{\nu} (\log(\lambda)) > 0.
\end{gather*}
This case is similar to Section \ref{sec:aneg}. \emr{However, it could not be reduced to it by replacing $f$ with $f^{-1}$, since, even if $f$ is invertible, the distributions of jumps $x_j - \lambda(w_{j - 1}) x_{j - 1}$ does not have to be persistent under this change.}

We use the same notation as in previous subsection. Let us fix $C, k$ from Property~\textbf{II} and constant
$\gamma_0$ defined by \eqref{gamma_ineq}.
Let us consider a pseudotrajectory $(w_j, x_j)_{j = 0}^N$. We introduce $A_j$, $r_j$ using equalities \eqref{def_A_j} and \eqref{def_r_j}.
Note, that according to Lemma \ref{F_lemma} shadowing problems for pseudotrajectories with same $A_j$ and $r_j$ are equivalent, hence we can fix one arbitrary point $x_j$ from $(x_j)_{j = 0}^N$ and define the rest of them using $\lambda(w_j)$ and $r_j$. It will be convenient for us to fix $x_N = 0$. We define $S_1$ using \eqref{eqS1}.
Replacing (\ref{old_ineq}) with
\begin{gather*}
    \dfrac{e^{A_q}}{e^{A_p} + e^{A_q}} \le 1
\end{gather*}
and arguing similarly to Section \ref{sec:aneg}, we conclude
\begin{gather}
    \label{new_ineq_2}
    1 - \p (S_1) \le \p \left( \exists\ 0 \le p < q \le N : \sum_{j = p + 1}^q e^{(A_p - A_j)} > (N + 1)^{{\gamma_0}} \right).
\end{gather}
Using the fact that $\nu$ is an invariant measure for the projection of Markov chain $\mathbb{M}_d$ on the first coordinate (Property \textbf{III}) and taking $n = (q - (p + 1))$ we get
\begin{gather*}
    \p \left( \exists\ 0 \le p < q \le N : \sum_{j = p + 1}^q e^{(A_p - A_j)} > (N + 1)^{{\gamma_0}} \right) \le \\
    \le (N + 1) \p \left(\exists\ n \le N : \sum_{j = 0}^n e^{(A_0 - A_j)} > (N + 1)^{{\gamma_0}} \right).
\end{gather*}

Note that $A_0 = 0$. We proceed with the following computations:
\begin{gather*}
    (N + 1) \p \left(\exists\ n \le N : \sum_{j = 0}^n e^{-A_j} > (N + 1)^{{\gamma_0}} \right) \le \\
    \le (N + 1) \p \left(\sum_{j = 0}^N e^{-A_j} > (N + 1)^{{\gamma_0}} \right) 
    \le (N + 1) \sum_{j = 0}^N \p \left( e^{- A_j} > (N + 1)^{{\gamma_0} - 1} \right) = \\
    = (N + 1) \sum_{j = 0}^N \p \left( \left|\dfrac{A_j}{j} - a \right| > \dfrac{({\gamma_0} - 1) \log(N + 1)}{j} + a \right).
\end{gather*}

After that we repeat the proof of the fact that 
\begin{gather*}
    \p(S_1) \to_{N \to \infty} 1
\end{gather*}
from the previous subsection, replacing $a$ with $-a$. 

Now we take $y_j$ from $S_1$, $w$ defined in Property \textbf{I} and introduce $z_j$ in the following way:
\begin{gather*}
    z_N = y_N, \quad
    z_{j} = \dfrac{z_{j + 1}}{\lambda(T^j (w))}.
\end{gather*}
We define $A_j'$ by \eqref{def_A_j'} and $S_2$ by \eqref{eqS2}.
Arguing similarly to Section \ref{sec:aneg} we conclude that $\p(S_2) \to_{N \to \infty} 0$ and repeat the rest of the proof unchanged. 



\section{Applications} \label{applications}

\subsection{Bernoulli shift} \label{bernoulli}

Consider the metric space $\Sigma = \{0, 1\}$ with $\dist_{\Sigma}(0, 1) = 1$ and probability measure $\mu(\{0\}) = \mu(\{1\}) = \dfrac{1}{2}$. Let $X$ be the space $\Sigma ^ {\Z}$ with standard topology and probability measure $\nu$ that arises from the product structure. Note that the topology on~$X$ is generated by the metric
\begin{gather*}
    \dist_X (\{w^{(j)}\}_{j \in \Z}, \{\Tilde{w}^{(j)}\}_{j \in \Z}) = \dfrac{1}{2^k}, \quad \text{where } k = \min \{|j|:w^{(j)} \ne \Tilde{w}^{(j)}\}.
\end{gather*}
On the space $X$ let us define the shift map
\begin{gather*}
    T: X \to X; \quad
    T(\{w^{(i)}\}_{i \in \Z})_j = w^{(j + 1)}.
\end{gather*}

Now we will show that Properties \textbf{I}, \textbf{II}, and \textbf{III} hold for described setting. Property \textbf{III} is satisfied because $\nu$ is preserved under both operations: shift $T$ and taking a random point (with respect to the normalized measure $\nu$) in a ball of arbitrary fixed radius, whose center is distributed with respect to $\nu$. Property \textbf{I} can be proved in the following way:

\begin{lm} \label{shad_distr}
    Consider the sequence $(w_j)_{j = 0}^{N}$ that is a $d$-pseudotrajectory of $T$ in the space $X$ taken at random with respect to the projection of the Markov chain $\mathbb{M}_d$ on $X$. Then we can define $w = w \left( (w_j)_{j = 0}^N \right) \in X$ such that
    \begin{gather*}
        dist(w_j, T^j(w)) < 2 d,
    \end{gather*}
    and $w$ has the distribution equal to $\nu$.
\end{lm}

\begin{proof}
    Let us consider integer $n$ such that $\dfrac{1}{2^{n + 1}} < d \le \dfrac{1}{2^n}$. Define the map $w((w_j)_{j = 0}^{N})$ as follows:
    \begin{gather*}
        w^{(j)} = \begin{cases}
            w_0^{(j)}, \text{for $j < n$ and $j \ge n + N$,} \\
            w_p^{(n)}, \text{for $j = n + p$, where $0 < p < N$.}
        \end{cases}
    \end{gather*}
    It is clear that $dist(w_j, T^j(w)) \le \dfrac{1}{2^n} < 2d$. Now we need to show that the distribution of $w$ equals to $\nu$, which fact is clear because $\p(w^{(j)} = 0 | w^{(i)} = x_i, i < j) = \p(w^{(j)} = 1 | w^{(i)} = x_i, i < j) = \dfrac{1}{2}$.
\end{proof}

To obtain Property \textbf{II} we will need the following classical lemma about probability of large deviations:

\begin{lm}
    \label{basic_LD}
    Let $(X_j)_{j \in \N}$ be a sequence of i.i.d. random variables. 
    \emr{
    Slightly abusing notation
    throughout this Lemma we denote by $\p$ the joint distribution of $(X_j)_{j \in \N}$ and by $\E$ the expected value with respect to said distribution.} Assume that $X_j$ has only finitely many possible values:
    \begin{gather}
        \label{X_j_distr}
        X_j = v_i \text{ with probability $p_i$, where $i \in \{1, \dots, q\}$}.
    \end{gather}
    We denote by $S_n$ the partial sum of first $n$ elements of the sequence $(X_j)_{j \in \N}$:
    \begin{gather*}
        S_n = X_1 + X_2 + \dots + X_n.
    \end{gather*}
    Then there exist positive constants $C, k > 0$ (these constants depend on the distribution of $X_j$ but not $n$), such that the following inequality holds:
    \begin{gather}
        \label{basic_LD_ineq}
        \p \left( \left|\dfrac{S_n}{n} - \E(X_1)\right| > \varepsilon \right) \le C \cdot e^{- k n \varepsilon^2}.
    \end{gather}
\end{lm}

Proof of this lemma can be found in \cite{Hoeffding}.

Since $\lambda$ is continuous there exists an integer $t \in \N$ such that $w^{(j)} = \Tilde{w}^{(j)}$ with $j \in \{-t, \dots, t\}$ implies 
\begin{gather*}
    |\log(\lambda(w)) - \log(\lambda(\Tilde{w}))| \le \dfrac{|a|}{4}.
\end{gather*}
Consider $\Theta$ -- the set of all continuous functions $\psi : X \to [\min_X(\lambda), \max_X(\lambda)]$, such that $\psi(w)$ depends only on $w^{(j)}$ for $j \in \{-t, \dots, t\}$ and $|\log(\psi(w)) - \log(\lambda(w))| \le -\frac{a}{4}$. $\Theta$ can be naturally identified with a nonempty compact set in $\R^{2^{2t + 1}}$. Note that $\E_{\nu} \log(\psi)$ is a continuous function on $\Theta$ and $\inf_{\psi \in \Theta} \E_{\nu} \log(\psi) \le a \le \sup_{\psi \in \Theta} \E_{\nu} \log(\psi)$. Now it follows that there exists a function $\Tilde{\lambda} \in \Theta$ such that
\begin{gather}
    \label{tilde_lambda_ineq}
    |\log(\lambda(w)) - \log(\Tilde{\lambda}(w))| < \dfrac{|a|}{4}, \quad \forall w \in X
\end{gather}
and
\begin{gather*}
    \E_{\nu} \log(\Tilde{\lambda}) = \E_{\nu} \log(\lambda) = a.
\end{gather*}
From now on we fix such $\Tilde{\lambda}$.
\begin{lm} \label{LD_tilde_lm}
    There exist constants $\Tilde{C}, \Tilde{k} > 0$ such that if we take a random point $w \in X$ distributed with respect to $\nu$ and $\hat{A}_j$ given by formula:
    \begin{gather*}
        \hat{A}_j = \sum_{p = 0}^{j - 1} \log (\Tilde{\lambda}(T^p(w))),
    \end{gather*}
    then 
    \emr{\begin{gather}
        \label{LD_tilde_lambda}
        \nu \left( \bigg \{ w \in X : \left| \dfrac{\hat{A}_j}{j} - a \right| > \varepsilon \bigg \} \right) < \Tilde{C} e^{- \Tilde{k} \varepsilon^2 j} \quad \forall \varepsilon > 0, j \in \N.
    \end{gather}}
\end{lm}

\begin{proof}
    Random variables $\Tilde{\lambda}(w), \Tilde{\lambda}(T^{2t + 1} w), \dots, \Tilde{\lambda}(T^{p(2t + 1)} w)$ are independent since elements $w^{(i)}$ and $w^{(j)}$ are independent if $i \ne j$ and $\Tilde{\lambda}$ depends only on the coordinates with indices in $\{-t, \dots, t\}$. Let us consider the partial sums that are given by the following formula:
    \begin{gather*}
        \hat{A}_j^{(q)} = \sum_{p = 0}^{\left[ \frac{j - q}{2t + 1} \right]} \log \left( \Tilde{\lambda} \left( T^{p (2t + 1) + q}(w) \right) \right), \quad \text{where $q \in \{0, \dots, 2t\}$.}
    \end{gather*}
    Note that for all $q$ the values of $\hat{A}_j^{(q)}$ are sums of i.i.d. random variables and function $\Tilde{\lambda}$ has only finite number of possible values, hence a large deviation principle from Lemma \ref{basic_LD} is applicable, so there are constants $C_0, k_0$, such that:
    \emr{\begin{gather*}
        \nu \left( \bigg \{ w \in X : \left| \dfrac{\hat{A}_j^{(q)}}{[\frac{j - q}{2t + 1}]} - a \right| > \varepsilon \bigg \} \right) < C_0 e^{- k_0 \varepsilon^2 [\frac{j - q}{2t + 1}]} \quad \forall \varepsilon > 0, j \in \N.
    \end{gather*}}
    Also observe that 
    \emr{\begin{gather*}
        \nu \left( \bigg \{ w \in X : \left| \dfrac{\hat{A}_j}{j} - a \right| > \varepsilon \bigg \} \right) \le \sum_{q = 0}^{2t} \nu \left( \bigg \{ w \in X : \left| \dfrac{\hat{A}_j^{(q)}}{[\frac{j - q}{2t + 1}]} - a \right| > \varepsilon \bigg \} \right),
    \end{gather*}}
    hence if we take constant $\Tilde{k} = \frac{k_0}{4t + 2}$ and $\Tilde{C}$ large enough, the inequality (\ref{LD_tilde_lambda}) will hold.
\end{proof}

Finally, we are prepared to prove that Property \textbf{II} holds for Bernoulli shift. 

Recall, that 
\begin{gather*}
    A_j' = \sum_{p = 0}^{j - 1} \log \left( \lambda(T^p (w)) \right).
\end{gather*}
for a random point $w \in X$ distributed with respect to measure $\nu$. We want to prove that for some constants $C, k > 0$ formula (\ref{property_2}) holds.

Let us note that due to (\ref{tilde_lambda_ineq}) we have
\begin{gather*}
    \left| \frac{A_j'}{j} - \frac{\hat{A}_j}{j} \right| < \frac{|a|}{4},
\end{gather*}
hence
\emr{\begin{gather*}
    \nu \left( \bigg \{ w \in X : \left| \dfrac{A_j'}{j} - a \right| > \frac{|a|}{2} \bigg \} \right) \le \nu \left( \bigg \{ w \in X : \left| \dfrac{\Tilde{A}_j}{j} - a \right| > \frac{|a|}{4} \bigg \} \right).
\end{gather*}}
By Lemma \ref{LD_tilde_lm} we have
\emr{\begin{gather*}
    \nu \left( \bigg \{ w \in X : \left| \dfrac{\Tilde{A}_j}{j} - a \right| > \frac{|a|}{4} \bigg \} \right) < \Tilde{C} e^{-\frac{\Tilde{k} a^2 j}{4}}.
\end{gather*}}
We take $C = \Tilde{C}$ and $k = \frac{\Tilde{k} a^2}{4}$ and the proof is finished. Now we have checked Properties \textbf{I}, \textbf{II}, and \textbf{III}, so we have shown that Theorem \ref{skew_prod_thm} can be applied to a skew product over Bernoulli shift. 

\subsection{Doubling map on a circle}
\label{chap:3}

Consider $X = \R / \Z \cong \T^1$ -- a circle, that is parametrized by a segment $[0, 1)$. For $w_1, w_2 \in X$ we will use $|w_1 - w_2|$ to denote the length of the shortest arc between $w_1$ and $w_2$. We consider map 
\begin{gather*}
    T: X \to X; \quad 
    T(w) = 2w \mod 1.
\end{gather*}
Measure $\nu$ is the Lebesgue measure on the circle. As always, we assume that 
\begin{gather}
    \label{2x_lambda_exp}
    a = \E_{\nu}(\log(\lambda)) \ne 0.
\end{gather}

Once again, we need to check Properties \textbf{I}, \textbf{II}, and \textbf{III} for $(X, T, \nu)$. Property \textbf{III} is straightforward, because Lebesgue measure is invariant for the doubling map, and taking a random point in a ball of arbitrary fixed radius, whose center is distributed with respect to Lebesgue measure preserves Lebesgue measure as well.

The following lemma proves that Property \textbf{I} is satisfied.

\begin{lm} \label{shad_distr_2x}
    Consider the sequence $(w_j)_{j = 0}^{N - 1}$ that is a $d$-pseudotrajectory of $T$ in the space $X$ taken at random with respect to the projection of the Markov chain $\mathbb{M}_d$ on $X$. Then we can define $w = w \left( (w_j)_{j = 0}^N \right) \in X$ such that
    \begin{gather*}
        dist(w_j, T^j(w)) < d,
    \end{gather*}
    and $w$ has the distribution equal to $\nu$.
\end{lm}

\begin{proof}
    First, note that $w_0$ is distributed with respect to $\nu$ (which is just a Lebesgue measure). Let us denote by $r_j$ the deviation of our pseudotrajectory on the $j$-th step:
    \begin{gather*}
        r_j = w_j - 2 w_{j - 1}, \quad j \in \{1, \dots, N\}.
    \end{gather*}
    We choose $r_j$ such that $|r_j| \le d$. Note, that according to our construction $r_j$ are distributed uniformly on an interval $(-d, d)$. Consider 
    \begin{gather*}
        w = \left( w_0 + \sum_{j = 1}^N 2^{-j} r_j \right) \mod 1.
    \end{gather*}
    Let us first prove that 
    \begin{gather*}
        |T^p(w) - w_p| < d.
    \end{gather*}
    To do so let us first prove by induction that
    \begin{gather*}
        T^p(w) = \left( w_p + \sum_{j = p + 1}^N 2^{p - j} r_j \right) \mod 1.
    \end{gather*}
    For $p = 0$ this formula coincides with the definition of $w$. Assume that it is true for $p$. Then
    \begin{gather*}
        T^{p + 1}(w) = \left(2 \cdot \left(w_p + \sum_{j = p + 1}^N 2^{p - j} r_j \right) \right) \mod 1 = \\
        = \left(2 w_p + r_{p + 1} + \sum_{j = p + 2}^N 2^{p + 1 - j} r_j \right) \mod 1 
        = \left(w_{p + 1} + \sum_{j = p + 2}^N 2^{p + 1 - j} r_j \right) \mod 1.
    \end{gather*}
    Now it is easy to see that 
    \begin{gather*}
        |T^p(w) - w_p| = \left| \sum_{j = p + 1}^N 2^{p - j} r_j \right| \le \sum_{j = p + 1}^N 2^{p - j} d < d.
    \end{gather*}
    It remains to say that $w_0$ and $r_1, r_2, \dots, r_N$ are independent, hence the distribution of $w$ is just a shift of Lebesgue measure $\nu$, which is again a Lebesgue measure $\nu$.
\end{proof}

Arguing similarly to the case of Bernoulli shift we construct the function $\Tilde{\lambda}$. First, since $\log(\lambda)$ is continuous, there exists a number $t \in \N$ such that if $|w - \Tilde{w}| < 2^{-t}$ then 
\begin{gather*}
    \left| \log(\lambda(w)) - \log(\lambda(\Tilde{w})) \right| < \dfrac{|a|}{4}.
\end{gather*}
We consider the set $\Theta$ of functions $\psi : X \to [\min_X(\lambda), \max_X(\lambda)]$, such that $\psi(w)$ is constant on the segments $[\frac{k}{2^t}, \frac{k + 1}{2^t})$ for $k \in \{0, 2^t - 1\}$, 
\begin{gather*}
    \left| \log(\psi(w)) - \log(\lambda(w)) \right| \le \frac{|a|}{4} \quad \forall w \in X.
\end{gather*}
$\Theta$ can be identified with a nonempty compact set in $\R^{2^{t}}$. Since $\E_{\nu} \log(\psi)$ is a continuous function on $\Theta$ (since $\min_X(\lambda) > 0$) and $\inf_{\psi \in \Theta} \E_{\nu} \log(\psi) \le a \le \sup_{\psi \in \Theta} \E_{\nu} \log(\psi)$, we obtain that there exists $\Tilde{\lambda} \in \Theta$, such that
\begin{gather}
    \label{2x_tilde_ineq_1}
    \left| \log(\lambda(w)) - \log(\Tilde{\lambda}(w)) \right| < \dfrac{|a|}{4}, \quad \forall w \in X,
\end{gather}
and 
\begin{gather*}
    \E_{\nu} \log \left( \Tilde{\lambda} \right) = \E_{\nu} \log \left( \lambda \right) = a.
\end{gather*}

Another important observation is the following one: for a random point $w$ that has distribution $\nu$ the random variables $\log \left( \Tilde{\lambda}(w) \right)$, $\log \left( \Tilde{\lambda}(T^t(w)) \right)$, \dots, $\log \left( \Tilde{\lambda}(T^{pt}(w)) \right)$ will be i.i.d. (because $\Tilde{\lambda}$ depends only on first $t$ digits after point in the binary numerical system) with finite number of possible values. We will use following notation:
\begin{gather*}
    \hat{A}_j = \sum_{p = 0}^{j - 1} \log \left( \Tilde{\lambda}(T^p(w)) \right).
\end{gather*}
The proof of Lemma \ref{LD_tilde_lm} can be repeated without any changes, so we will only give its formulation.

\begin{lm} \label{2x_LD_tilde_lm}
    There exist constants $C, k > 0$ such that if we take a random point $w \in X$ distributed with respect to $\nu$ and $\hat{A}_j$ given by formula:
    \begin{gather*}
        \hat{A}_j = \sum_{p = 0}^{j - 1} \log \left( \Tilde{\lambda} \left( T^p(w) \right) \right),
    \end{gather*}
    then 
    \emr{\begin{gather*}
        \nu \left( \bigg \{ w \in X : \left| \dfrac{\hat{A}_j}{j} - a \right| > \varepsilon \bigg \} \right) < C e^{- k \varepsilon^2 j} \quad \forall \varepsilon > 0, j \in \N.
    \end{gather*}}
\end{lm}

The rest of the proof from Subsection \ref{bernoulli} can be repeated without any changes, and we obtain that Theorem \ref{skew_prod_thm} holds for skew products over the Doubling map on a circle.

\subsection{Linear hyperbolic map on a torus}

In this subsection we will consider a skew product over a linear hyperbolic map on a torus and will apply our technique to establish shadowing property for it. We will use the same notations as in the previous part.

Let us consider $X = \T^m = \R^m/\Z^m$ and map
\begin{gather*}
    T: X \to X, \quad T(w) = A w,
\end{gather*}
where $A$ is a hyperbolic matrix with integer entries and $|\det(A)| = 1$. We will denote by $E^s$ and $E^u$ stable and unstable subspaces of the action of $A$ on $\R^m$ and by $S$ and $U$ the projections on $E^s$ along $E^u$ and on $E^u$ along $E^s$. Since $A$ is hyperbolic, we know that 
\begin{gather*}
    S + U = \id.
\end{gather*}
We will also need to fix constants $B > 0$ and $t \in (0, 1)$, such that for all $n \in \N$
\begin{gather}
    |A^n v| < B t^n |v|, \quad \text{if $v \in E^s$;}\\
    |A^{-n} v| < B t^n |v|, \quad \text{if $v \in E^u$.}
\end{gather}
As usual,
\begin{gather*}
    f: X \times \R \to X \times \R, \\
    f(w, x) = (T(w), \lambda(w) x).
\end{gather*}
and $\lambda: X \to \R$ is a continuous function. Measure $\nu$ is a Lebesgue measure on $\T^m$. Once again, we assume that 
\begin{gather*}
    \E_{\nu}(\log(\lambda)) \ne 0.
\end{gather*}

We will show that Properties \textbf{I}, \textbf{II}, and \textbf{III} holds for this system. Property \textbf{III} is once again proved by the observation that Lebesgue measure is invariant for $T$ since $\det(A) = 1$ and it is also invariant with respect to taking a random point in a ball of arbitrary fixed radius, if the center of the ball is distributed with respect to Lebesgue measure.

We prove that Property \textbf{I} holds with the following lemma.

\begin{lm} \label{shad_distr_cat_map}
    Consider the sequence $(w_j)_{j = 0}^{N}$ that is a $d$-pseudotrajectory of $T$ in the space $X$ taken at random with respect to the projection of the Markov chain $\mathbb{M}_d$ on $X$. Then we can define $w = w \left( (w_j)_{j = 0}^{N} \right) \in X$ such that
    \begin{gather*}
        dist(w_j, T^j(w)) < \frac{2}{1 - t} d,
    \end{gather*}
    and $w$ has the distribution equal to $\nu$.
\end{lm}

\begin{proof}
    First of all, let us introduce 
    \begin{gather*}
        r_{j + 1} = w_{j + 1} - A w_j, \quad \text{for $j \in \{0, \dots, N - 1\}$}.
    \end{gather*}
    Because of the fact that the sequence $(w_j)_{j = 0}^{N}$ is a $d$-pseudotrajectory we can define $r_j$ in such a way that $|r_j| < d$.
    Consider 
    \begin{gather}
        \label{w_def_cats}
        w = w_0 + \sum_{j=1}^{N} A^{-j} U r_j.
    \end{gather}
    Straightforward computation shows that
    \begin{gather}
        \label{recursion_fla}
        T^k(w) = w_k - \sum_{j = 1}^k A^{k - j} S r_j + \sum_{j = k + 1}^N A^{k - j} U r_j.
    \end{gather}
    Now we can estimate $|T^k(w) - w_k|$ in a following way:
    \begin{gather*}
        |T^k(w) - w_k|  
        \le \sum_{j = 1}^{k + 1} |A^{k + 1 - j} S r_j| + \sum_{j = k + 2}^N |A^{k + 1 - j} U r_j| < \frac{2}{1 - t} d.
    \end{gather*}
    Now we need to prove that $w$, defined by (\ref{w_def_cats}) is distributed with respect to $\nu$. It is clear that $w_0$ is distributed with respect to $\nu$ and $(r_j)_{j = 1}^N$ is independent with $w_0$. Hence $w$ is distributed with respect to $\nu$, because $\nu$ is a Lebesgue measure, and Lebesgue measure is translation-invariant.
\end{proof}

In order to check Property \textbf{II} we will use theory of large deviatons for unformly hyperbolc systems developed in \cite{Orey}.

\begin{thm}
    \label{OreyPelican}
    Assume $T : M \to M$ is a transitive Anosov diffeomorphism of a manifold $M$ and $\nu$ is a Lebesgue measure. Then for each continuous function $\varphi$ there is a lower semicontinuous $K^{\varphi}: \R \to [0, \infty]$ which satisfies the following properties:
    \begin{enumerate}
        \item $
            \limsup_{n \to \infty}{n ^ {-1}} \log \nu \{ x : \frac{1}{n} \sum_{0}^{n - 1} \varphi (T^j (x)) \in A \} \le - \inf\{ K^{\varphi} (t): t \in A \},
        $
        
        for closed $A \subset \R$;
        
        \item $
            \liminf_{n \to \infty}{n ^ {-1}} \log \nu \{ x : \frac{1}{n} \sum_{0}^{n - 1} \varphi (T^j (x)) \in A \} \ge - \inf\{ K^{\varphi} (t): t \in A \},
        $
        
        for open $A \subset \R$.
    \end{enumerate}
    
    Moreover, $K^{\varphi}$ is defined by the following formula (here $\mathcal{M}$(M) stands for the set of all Borel probability measures on $M$):
    \begin{gather}
        K^{\varphi} (t) = \inf \left\{ K(m): m \in \mathcal{M} (M) \text{ and } \int \varphi \dd m = t \right\},
    \end{gather}
    where
    \begin{gather}
        K(m) = 
        \begin{cases}
            -h_{m} (T) + \int \log \left| \det \left( \D T|_{E^u} \right) \right|\dd m, \quad \text{if $m$ is invariant under $T$,} \\
            + \infty, \quad \text{otherwise,}
        \end{cases}
    \end{gather}
    where $h_{m}(T)$ is the metric entropy of $T$.
\end{thm}

Now we will apply this theorem to our setting. According to Pesin entropy formula we have

\begin{gather*}
    h_{\nu} (T) = \int \log \left| \det \left( \D T|_{E^u} \right) \right|\dd \nu,
\end{gather*}
where $\nu$ is a Lebesgue measure. It is also well known that $\nu$ is a unique measure of maximal entropy for $T$ (see \cite[Theorem 20.1.3]{KH}), hence $K(m) \ge 0$ and 
\begin{gather*}
    K(m) = 0 \Leftrightarrow m = \nu.
\end{gather*}
We consider $A = \R \setminus \left( a - \frac{|a|}{4}, a + \frac{|a|}{4} \right)$ and $\varphi = \log(\lambda)$, and by Theorem \ref{OreyPelican} we have 
\begin{gather*}
    \limsup_{n \to \infty}{n ^ {-1}} \log \nu \left\{ x : \left| \frac{1}{n} \sum_{0}^{n - 1} \varphi (T^j (x)) - a \right| \ge \frac{|a|}{4} \right\} \le - \inf \left\{ K^{\varphi} (t): |t - a| \ge \frac{|a|}{4} \right\}.
\end{gather*}
Let us prove that 
\begin{gather}
    \label{nonzero_exp}
    \inf \left\{ K^{\varphi} (t): |t - a| \ge \frac{|a|}{4} \right\} > 0.
\end{gather}
Note that function $K$ is lower semi-continuous with respect to weak-* topology on the space of measures, because entropy is upper semi-continuous. According to the definition of $K^{\varphi} (t)$ we have
\begin{gather*}
    \inf \left\{ K^{\varphi} (t): |t - a| \ge \frac{|a|}{4} \right\} = \inf \left\{ K(m): m \in \mathcal{M} (M) \text{ and } \left| \int \varphi \dd m - a \right| \ge \frac{|a|}{4}  \right\}.
\end{gather*}
Let us note that the set 
\begin{gather*}
    \left\{ m \in \mathcal{M} (M): \left| \int \varphi \dd m - a \right| < \frac{|a|}{4}  \right\}
\end{gather*}
is open in weak-* topology. Hence, the set
\begin{gather*}
    \left\{ m \in \mathcal{M} (M): \left| \int \varphi \dd m - a \right| \ge \frac{|a|}{4}  \right\} = \mathcal{M} (M) \setminus \left\{ m \in \mathcal{M} (M): \left| \int \varphi \dd m - a \right| < \frac{|a|}{4}  \right\}
\end{gather*}
is compact and the lower semi-continuous function $K$ attains its infimum on that set. Since $\nu$ is the only global minimum for $K$, we have
\begin{gather*}
    \inf \left\{ K^{\varphi} (t): |t - a| \ge \frac{|a|}{4} \right\} = \min \left\{ K^{\varphi} (t): |t - a| \ge \frac{|a|}{4} \right\} > 0,
\end{gather*}
which proves (\ref{nonzero_exp}).
Hence
\emr{\begin{gather*}
    \nu \left( \bigg \{ w \in X : \left| \dfrac{A_j'}{j} - \E_{\nu}(\log(\lambda)) \right| > \frac{|a|}{2} \bigg \} \right) < C e^{- k j} \quad \forall j \in \N
\end{gather*}}
with appropriate $C, k > 0$ and Property \textbf{II} holds, which implies that
Theorem \ref{skew_prod_thm} can be applied in this case.


\end{document}